\documentclass[sn-mathphys]{sn-jnl}

\jyear{2022}%

\theoremstyle{thmstyleone}%
\newtheorem{theorem}{Theorem}
\theoremstyle{thmstylethree}%

\newtheorem{observation}{Observation}%

\theoremstyle{thmstylethree}%
\newtheorem{corollary}[theorem]{Corollary}

\begin{document}

\title[Connected power domination number of product graphs]{Connected power domination number of product graphs}


\author*{ \sur{S. Ganesamurthy}}\email{ganesamurthy66@gmail.com}

\author{\sur{J. Jeyaranjani}}\email{jeyaranjani.j@gmail.com}
\equalcont{These authors contributed equally to this work.}

\author{\sur{R. Srimathi}}\email{gsrimathi66@gmail.com}
\equalcont{These authors contributed equally to this work.}

\affil*[1]{\orgdiv{Department of Mathematics}, \orgname{Periyar University}, \orgaddress{\city{Salem}, \postcode{636011}, \state{Tamil Nadu}, \country{India}}}

\affil[2]{\orgdiv{Department of Computer science and Engineering}, \orgname{Kalasalingam Academy of Research and Education}, \orgaddress{\street{ Krishnankoil}, \city{Srivilliputhur}, \postcode{626128}, \state{Tamil Nadu}, \country{India}}}

\affil[3]{\orgdiv{Department of Mathematics}, \orgname{Idhaya College of Arts and Science for Women}, \orgaddress{\city{Lawspet}, \postcode{605008}, \state{Puducherry}, \country{India}}}


\abstract{In this paper, we consider the connected power domination number ($\gamma_{P, c}$) of three standard graph products. The exact value for $\gamma_{P, c}(G\circ H)$ is obtained for any two non-trivial graphs $G$ and $H.$ Further, tight upper bounds are proved for the connected power domination number of the Cartesian product of two graphs $G$ and $H.$ Consequently, the exact value of the connected power domination number of the Cartesian product of some standard graphs is determined. Finally, the connected power domination number of tensor product of graphs is discussed.}




\maketitle

\section{Introduction}

We only consider non-trivial simple connected graphs of finite order, unless otherwise stated. For a vertex $v\in V(G),$ the \textit{open neighborhood} of $v$ is $N(v)=\{u\,:\,uv\in E(G)\}$ and the \textit{closed neighborhood} of $v$ is $N[v]=\{v\}\cup N(v).$ For a set $A\subset V(G),$ the \textit{open neighborhood of $A$} is $N(A)= \cup_{v\in A} N(v)$  and the \textit{closed neighborhood of $A$} is $N[A]=\cup_{v\in A} N[v].$ The subgraph of the graph $G$ induced by the subset $A$ of the vertices of $G$ is denoted by $\langle A \rangle.$ A vertex $v\in V(G)$ is called \textit{universal vertex} of $G$ if $v$ is adjacent to each vertex of the graph $G.$

Let $K_n,\,P_n,\,C_n,\,W_n,\,F_n,$ and $K_{m,\,n},$ respectively, denote complete graph, path, cycle, wheel, fan, and complete bipartite graph. For $k\geq 3$ and $1\leq m_1\leq m_2\leq \dots\leq m_k,$ the complete multipartite graph with each partite set of size $m_i$ is denoted by $K_{m_1,\,m_2,\,\dots,\,m_k}.$ 

Let $S\subset V(G).$ If $N[S]=V(G), $ then $S$ is called a \textit{domination set}. If the subgraph induced by the dominating set is connected, then we say $S$ is a \textit{connected dominating set}. For each vertex $v\in V(G),$ if a dominating set $S$ satisfies the property $N(v) \cap S \neq \emptyset,$ then we call the set $S$ is a \textit{total dominating set}. The minimum cardinality of dominating set (connected dominating set) of $G$ is called domination number (connected domination number) and  it is denoted by $\gamma(G)$ ($\gamma_c(G)$).

\emph{\textbf{Algorithm:}}\cite{dmks22} For the graph $G$ and a set $S\subset V(G),$ let $M(S)$ be the collection of vertices of $G$ monitored by $S.$ The set $M(S)$ is built by the following rules:
\begin{enumerate}
	\item (Domination)
	\item[] Set $M(S) \leftarrow S\cup N(S).$
	\item (Propagation)
	\item[] As long as there exists $v\in M(S)$ such that $N(v)\cap (V(G)-M(S))=\{w\},$ set $M(S)\leftarrow M(S)\cup \{w\}.$
\end{enumerate}

In other words, initially the set $M(S)=N[S],$ and then repeatedly add to $M(S)$ vertices $w$ that has a neighbor $v$ in $M(S)$ such that all the other neighbors of $v$ are already in $M(S).$ After no such vertex $w$ exists, the set monitored by $S$ is constructed. 

For a subset $S$ of $V(G),$ if $M(S)=V(G),$ then the set $S$ is called a \textit{power dominating set} (PDS). The minimum cardinality of power dominating set of $G$ denoted by $\gamma_{p}(G).$ If the subgraph of $G$ induced by the vertices of a PDS $S$ is connected, then the set $S$ is \textit{connected power domination set} (CPDS), and its minimum cardinality is denoted by $\gamma_{P,\,c}(G).$  

\noindent {\bf \cite{laa428} Color-change rule:} \textit{If $G$ is a graph with each vertex colored either white or black, $u$ is a black vertex of $G,$ and exactly one neighbor $v$ of $u$ is white, then change the color of $v$ to black. Given a coloring of $G,$ the derived coloring is the result of applying the color-change rule until no more changes are possible.} A \textit{zero forcing set} for a graph G is a set $Z\subset V (G)$ such that if initially the vertices in $Z$ are colored black and the remaining vertices are colored white, the entire graph G may be colored black by repeatedly applying the color-change rule. The zero forcing number of $G, Z(G),$ is the minimum cardinality of a zero forcing set. 

If a zero forcing set $Z$ satisfies the connected condition, then we call such set as \textit{connected zero forcing set} (CZFC) and it is denoted by $Z_c.$  The connected zero forcing number of $G, Z_c(G),$ is the minimum cardinality of a connected zero forcing set.

For a graph $G$ and a set $X \subseteq V(G),$ the set $X_i,\,i>0,$ denotes the collection of all vertices of the graph $G$ monitored by the propagation up to step $i,$ that is, $X_1=N[X]$ (dominating step) and $X_{i+1}=\cup\{N[v]\,:\, v\in X_i$ such that $\vert N[v]\setminus X_i\vert \leq 1\}$ (propagation steps). Similarly, for a connected zero forcing set $Z_c \subseteq V(G)$ and $i\geq 1,$ let $Z_c^i$ denote the collection of all vertices of the graph  $G$ whose color changed from white to black at step $i$ (propagation steps).

For two graphs $G$ and $H,$ the vertex set of the Cartesian product ($G\square H$), tensor product $(G\times H)$ and lexicographic product ($G\circ H$) is $V(G)\times V(H).$ The adjacency relationship between the vertices $u=(a,\,b)$ and $v=(x,\,y)$ of these products are as follows:
\begin{itemize}
	\item Cartesian product: $uv\in E(G\square H)$ if either $a=x$ and $by\in E(H),$ or $b=y$ and $ax\in E(G).$
	\item Tensor product: $uv\in E(G\times H)$ if $ax\in E(G)$ and $by\in E(H).$
	\item Lexicographic product: $uv\in E(G\circ H)$ if $ax\in E(G),$ or $a=x$ and $by\in E(H).$
\end{itemize}

Let $G \ast H$ be any of the three graph products defined above. Then the subgraph of $G \ast H$ induced by $\{g\}\times V(H)$ ($V(G)\times \{h\})$ is called an $H$-fiber ($G$-fiber) and it is denoted by $^gH$ ($G^h$). Notation and definitions which are not presented here can be found in \cite{rbbook,hikbook}.

The problem of computing the power domination number of $G$ is NP-hard in general. The complexity results for power domination in graphs are studied in \cite{ajco19,gnr52,hhhh15,lllncs}. Further, some upper bound for the power domination number of graphs is obtained in \cite{zkc306}. Furthermore, the power domination number of some standard families of graphs and product graphs are studied in \cite{bf58,bgpv38,dmks22,dh154,ks13,ks16,skp18,sk11,sk48,vthesis,vvlncs,vvh38}. Recently, Brimkvo et al. \cite{bms38} introduced the concept of connected power domination number of graph and obtained the exact value for trees, block graph, and cactus graph.  Further, in \cite{gplncs}, the complexity results for split graph, chain graph, and chordal graph are considered. 

In this paper, we extend the study of connected power domination number for three standard products.

\section{The Lexicographic Product}

The exact value of the power domination number of the lexicographic product of graphs obtained in \cite{dmks22}. In this section, we have obtained the exact value of the connected power domination number of $G\circ H.$ The assumption of the connected condition for graph $H$ is relaxed in this section.  
 \begin{theorem}
For any two graphs $G$ and $H,$ 

\begin{center}
$\gamma_{P,c}(G\circ H)= 
\left\{ \begin{array}{rl}
\mbox{$\gamma_c(G);$} & \mbox{ if $\gamma_c(G)\geq 2,$} \\ 
\mbox{$1;$} & \mbox{either $\gamma(G)=\gamma(H)=1$ or $\gamma(G)=1$ and $H\cong \overline{K_2},$}\\
\mbox{$2;$} & \mbox{if $\gamma(G)=1$ and $\gamma(H)>1$ with $\vert V(H)\vert\geq 3.$}
\end{array}\right.$
\end{center}
\end{theorem}

\begin{proof}
First we complete the proof for the case $\gamma_c(G)\geq 2.$ Let $X$ be a minimum connected dominating set of $G$ and let $u\in V(H).$ Set $S=X\times \{u\}.$ As $X$ is a connected dominating set of $G,$ it is a total dominating set of $G;$ consequently, each vertex of $G$ is a neighbor of some vertex in $X.$ Thus each vertex $(g,\,h)\in V(G\circ H)$ is a neighbour of some vertex in $S.$ Since $\langle S\rangle$ is connected and which monitors each vertex of $G\circ H,$ $\gamma_{P,c}(G\circ H)\leq \gamma_c(G).$

Assume that $S$ is a connected power dominating set of $G\circ H$ whose cardinality is strictly less than $\gamma_c(G).$ Then there exists a vertex $u\in V(G)$ such that $\{u\}\times V(H) \cap N[S]=\emptyset.$ Hence the vertices in $\{u\}\times V(H)$ are monitored by the propagation. Let $A= \{u\}\times V(H).$ Clearly, each vertex in $V(G\circ H)\setminus A$ has either zero or $\vert A\vert$ neighbours in $\langle A\rangle\cong \,^uH$-fiber. Therefore propagation on $^uH$-fiber is not possible as $\vert V(H)\vert\geq 2.$ Therefore $\gamma_{P,c}(G\circ H)\geq \gamma_c(G).$

 Let $\gamma(G)=\gamma(H)=1.$ Then the graphs $G$ and $H$ have universal vertices, namely, $u$ and $v,$ respectively. Consequently, the vertex $(u,\,v)\in V(G\circ H)$ is a universal vertex of the graph $G\circ H.$ Thus $\gamma_{P,c}(G\circ H)=1.$ 

Consider $\gamma(G)=1$ and $H\cong \overline{K_2}.$ Let $u$ be a universal vertex of $G$ and let $V(H)=\{x,\,y\}.$ Then the vertex $(u,\,x)\in V(G\circ H)$ dominates all the vertices of the graph $G\circ H$ except $(u,\,y).$ Clearly, the vertex $(u,\,y)$ is monitored by the propagation as $(u,\,y)$ is the only unmonitored vertex of $G\circ H.$ Therefore, $\gamma_{P,c}(G\circ H)=1.$ 

Assume that $\gamma(G)=1$ and $\gamma(H)>1.$ It is easy to observe that a $\gamma_{P,c}(G\circ H)\geq 2$ as $\vert V(H)\vert\geq 3$ and $\gamma(H)>1.$ Let $u$ be a universal vertex of the graph $G.$ Then the set $\{(u,\,a),\,(v,\,a)\}$ dominates all the vertices of the graph $G\circ H.$ Since $u$ is a universal vertex,  $\langle \{(u,\,a),\,(v,\,a)\}\rangle\cong K_2.$ Hence, $\gamma_{P,c}(G\circ H)\leq 2.$
\end{proof}

\section{The Cartesian Product}
We begin this section by proving a general upper bound for the connected power domination number of $G\square H.$

\begin{theorem}
For any two graphs $G$ and $H,$ 
\begin{center}
$\gamma_{P,c}(G \,\square\,H)\leq$ min$\{\gamma_{P,c}(G)\vert V(H)\vert, \gamma_{P,c}(H)\vert V(G)\vert\}.$
\end{center}
\end{theorem}
\begin{proof}
Let $X$ be a CPDS of $G.$ Consider $X'=X\times V(H).$ Clearly, for each vertex $u\in X,\,^uH$-fiber is observed as $\{u\}\times V(H)\in X'.$ Also, by our choice of $X',$ for each vertex $v\in N(X),\,^vH$-fiber is observed (dominating step). To complete the proof, it is enough to show that if $w\in X_i,$ then $V(^wH)\in X_i'.$ We proceed with the proof by induction. The result is true for $i=1.$ Assume that the result holds for some $i>0.$ Let $w\in X_{i+1}.$ If $w\in X_i,$ then $V(^wH)\in X_i'$ by induction hypothesis. If $w\notin X_i,$ then there exists a vertex $y\in X_i$ which is the neighbour of $w$ such that $\vert N[y]\setminus X_i\vert\leq 1.$ This gives $V(^yH)\in X_i',$ by induction hypothesis. Hence, for fixed $h\in V(H),\,\vert N[(y,\,h)]\setminus X_i'\vert=\vert N[y]\setminus X_i\vert\leq 1.$ Thus, $N[(y,\,h)]\in X_{i+1}'$ which implies that $(w,\,h)\in X_{i+1}'.$ As it is true for each $h\in V(H),\, V(^wH)\in X_{i+1}'.$ Therefore, $\gamma_{P,c}(G \,\square\,H)\leq \gamma_{P,c}(G)\vert V(H)\vert.$ It is easy to prove that $\gamma_{P,c}(G \,\square\,H)\leq \gamma_{P,c}(H)\vert V(G)\vert$ as $G\square H$ is commutative.
\end{proof}

From the definitions of CPDS and CZFS, it is clear that if $X\subseteq V(G)$  is a CPDS, then $N[X]$ is a CZFS. From this observation, we prove the following upper bound for $\gamma_{P,c}(G\square H)$ in terms of the product of Connected zero forcing number and connected domination number.

\begin{theorem}\label{upcpdczfs}
For any two graphs $G$ and $H,$ 
\begin{center}
$\gamma_{P,c}(G \,\square\,H)\leq$ min$\{Z_c(G)\gamma_c(H), Z_c(H)\gamma_c(G)\}.$
\end{center}
\end{theorem}
\begin{proof}
Let $Z_c$ be a CPDS of $G$ and let $S$ be a connected dominating set of $H.$  Consider $X=Z_c\times S.$ Clearly, for each vertex $u\in Z_c,\,^uH$-fiber is observed as $\{u\}\times S\in X.$ We proceed with the proof by induction. The result is true for $i=0.$ Assume that the result holds for some $i\geq 0.$ Let $w\in Z_c^{i+1}.$ If $w\in Z_c^i,$ then $V(^wH)\in X_i$ by induction hypothesis. If $w\notin Z_c^i,$ then there exists a vertex $y\in Z_c^i$ which is the neighbour of $w$ such that $\vert N[y]\setminus Z_c^i\vert\leq 1.$ This gives $V(^yH)\in X_i,$ by induction hypothesis. Hence, for fixed $h\in V(H),\,\vert N[(y,\,h)]\setminus X_i\vert=\vert N[y]\setminus Z_c^i\vert\leq 1.$ Thus, $N[(y,\,h)]\in X_{i+1}$ which implies that $(w,\,h)\in X_{i+1}.$ As it is true for each $h\in V(H),\, V(^wH)\in X_{i+1}.$ Therefore, $\gamma_{P,c}(G \,\square\,H)\leq Z_c(G)\gamma_c(H).$ In a similar way, it is easy to prove that $\gamma_{P,c}(G \,\square\,H)\leq  Z_c(H)\gamma_c(G).$ 
\end{proof}

The upper bound in the above theorem is tight if $G$ has a universal vertex and $H\in\{P_n,\,C_n,\,W_n,\,F_n\}.$ Also, if we replace $Z_c=Z$ and $\gamma_c=\gamma$ in the above theorem, then we have the upper bound for $\gamma_P(G\square H)$ in terms of zero forcing number and domination number. 
\begin{corollary}
For any two graphs $G$ and $H,$ 
\begin{center}
$\gamma_{P}(G \,\square\,H)\leq$ min$\{Z(G)\gamma(H), Z(H)\gamma(G)\}.$
\end{center}
\end{corollary}  

The following corollaries are immediate from Theorem \ref{upcpdczfs} as $Z_c(P_n)=1,$ $Z_c(C_n)=2,$ $Z_c(W_n)=3$ and $Z_c(F_n)=2.$
\begin{corollary}
For a graph $G,$ $\gamma_{P,c}(G \,\square\,P_n)\leq \gamma_c(G).$
\end{corollary}
\begin{corollary}\label{cpdgboxcn}
For a graph $G,$ $\gamma_{P,c}(G \,\square\,C_n)\leq 2\gamma_c(G),$ where $\vert V(G)\vert\geq 3.$
\end{corollary}

\begin{corollary}\label{cpdgboxwn}
For $n\geq 4$ and a graph $G,\,\gamma_{P,c}(G \,\square\,W_n)\leq 3\gamma_c(G),$ where $\vert V(G)\vert\geq 3.$
\end{corollary}
 %

\begin{corollary}\label{cpdgboxfn}
For a graph $G,$ $\gamma_{P,c}(G \,\square\,F_n)\leq 2\gamma_c(G),$ where $\vert V(G)\vert\geq 3$ and $n\geq 3.$
\end{corollary}
As mentioned earlier, the upper bounds in the above four corollaries are tight if $G$ has a universal vertex. Some of their consequences are listed in the following table.
\begin{table}[!h]
\begin{center}
\begin{tabular}{ l l l }
\hline
Result & $G$ & $\gamma_{P,c}$ \\\hline
 Corollary \ref{cpdgboxcn} & $C_m\square K_n,\,m,\,n\geq 3 $& 2 \\ 
Corollary \ref{cpdgboxcn} & $C_m\square W_n,\,m\geq 3$ and $m\geq 4$ & 2 \\
Corollary \ref{cpdgboxcn} & $C_m\square K_{1,\,m},\,m,\,n\geq 3 $& 2 \\
Corollary \ref{cpdgboxcn} & $C_m\square F_n,\,m,\,n\geq 3 $& 2   \\
Corollary \ref{cpdgboxwn} & $W_m\square W_n,\,m,\,n\geq 4$ & 3 \\
Corollary \ref{cpdgboxwn} & $W_m\square K_{1,\,m},\,m,\,n\geq 4 $& 3 \\
Corollary \ref{cpdgboxwn} & $W_m\square K_n,\,m,\,n\geq 4$ & 3   \\
Corollary \ref{cpdgboxfn} & $F_m\square F_n,\,m,\,n\geq 3$ & 2 \\
Corollary \ref{cpdgboxfn} & $F_m\square K_n,\,m,\,n\geq 3$ & 2\\
Corollary \ref{cpdgboxfn} & $F_m\square  K_{1,\,n},\,m,\,n\geq 3$ & 2\\
Corollary \ref{cpdgboxfn} & $F_m\square W_n,\,m\geq 3$ and $n\geq 4$ &2\\\hline
\end{tabular}
\end{center}
\end{table}

 \begin{observation}\label{O1}
For any graph $G,$ $\gamma_p(G)\leq \gamma_{P,c}(G).$
\end{observation}

\begin{theorem}\cite{sk11}\label{pdofkmtimeskn}
For $2\leq m\leq n,$ $\gamma_p(K_m\square K_n)=m-1.$
\end{theorem} 

\begin{theorem}
For $2\leq m\leq n,$ $\gamma_{P,c}(K_m\square K_n)=m-1.$
\end{theorem}
\begin{proof}
By Theorem \ref{pdofkmtimeskn} and Observation \ref{O1}, we have $m-1\leq \gamma_{P,c}(K_m\square K_n).$ Let $V(K_m)=\{v_1,\,v_2,\,\dots,\,v_m\}$
and $V(K_n)=\{u_1,\,u_2,\,\dots,\,u_n\}.$ It is easy to observe that the set $S=\{(v_1,\,u_1),\,(v_2,\,u_1),\,\dots,\,(v_{m-1},\,u_1)\}$ is a CPDS of $K_m\square K_n.$ Thus, $\gamma_{P,c}(K_m\square K_n) = m-1$ as $\vert S\vert=m-1.$\end{proof}

\begin{theorem}\cite{ks16}\label{pdkmtimesk1,n}
For $m,\,n\geq 3,$ $\gamma_{P}(K_m\square K_{1,\,n})=min\{m-1,\,n-1\}.$
\end{theorem}
\begin{theorem}
For $m,\,n\geq 3,$ $\gamma_{P,c}(K_m\square K_{1,\,n})=min\{m-1,\,n\}.$
\end{theorem}
\begin{proof}
Let $V(K_m)=Z_m$ and $V(K_{1,n})=Z_{n+1},$ where the vertex $0$ is the universal vertex of $K_{1,\,n}.$ Then $V(K_m\square K_{1,\,n})=Z_m\times Z_{n+1}.$ 

\noindent {\bf Case 1:} $m\leq n+1$

By Theorem \ref{upcpdczfs}, we have $\gamma_{P,c}(K_m\square K_{1,\,n}) \leq m-1$ as $Z_c(K_m)=m-1$ and $\gamma_c(K_{1,\,n})=1.$ By Theorem \ref{pdkmtimesk1,n} and Observation \ref{O1}, $m-1\leq \gamma_{P,c}(K_m\square K_{1,\,n}).$ Hence, $\gamma_{P,c}(K_m\square K_{1,\,n})= m-1.$

\noindent {\bf Case 2:} $m>n+1$


 Since $\gamma(K_m)=1$ and $Z_c(K_{1,n})=n,\,\gamma_{P,c}(K_m\square K_{1,\,n}) \leq n$ (By Theorem \ref{upcpdczfs}). To prove the lower bound, first we need to observe that any minimum CPDS $X$ of $K_m\square K_{1,\,n}$ must contains at least one of the vertices of the form $(i,\,0)$ for some $i\in Z_m;$ otherwise, all the vertices in any CPDS $X \subset V(K_m^j),$ for some fixed $j,$ where $j\in (Z_m\setminus \{0\}),$ and hence $\vert X \vert >n$ as $m>n+1.$ Suppose there exists a minimum CPDS $X$ of $K_m\square K_{1,\,n}$ with $\vert X \vert \leq  n-1.$ Then the vertices in at least three $^iK_{1,\,n}$-fiber and two $K_m^j$-fiber do not belong to $X.$ WLOG let $i\in\{m-1,\,m,\,m+1\}$ and $j\in \{n-1,\,n\}.$  Let $A= \{(i,\,j)\,\vert\, i\in\{m-1,\,m,\,m+1\}\,\,\mbox{and}\,\,j\in \{n-1,\,n\} \}.$ Since $\vert N(x)\cap A\vert > 1$ for any vertex $x\notin X$ and $x\in N(A)\setminus A,$ propagation is not possible to observe any vertices in the set $A.$ This leads to the contradiction for the cardinality of the minimum CPDS is $n-1.$ Thus, $\gamma_{P,c}(K_m\square K_{1,\,n}) \geq n.$ This completes the proof.

From Case $1$ and $2,$ we have $\gamma_{P,c}(K_m\square K_{1,\,n})=min\{m-1,\,n\}.$
\end{proof}

\begin{theorem}
For $3\leq x\leq y,\,\gamma_{P,\,c}(K_{1,\,x}\square K_{1,\,y})=x.$
\end{theorem}

\begin{proof}
Let $V(K_{1,\,x})=Z_x$ and $V(K_{1,\,y})=Z_y.$ Consider the vertex with label $0$ is the universal vertex of the graph $K_{1,\,x}$ (respectively, $K_{1,\,y}$). By Theorem \ref{upcpdczfs}, we have $\gamma_{P,c}(K_{1,\,x}\square K_{1,\,y}) \leq x$ as $Z_c(K_{1,\,x})=x$ and $\gamma_c(K_{1,\,y})=1.$ 

To attain the lower bound, we claim that any set $X\subset V(K_{1,\,x}\square K_{1,\,y})$ with cardinality $x-1$ does not satisfy the CPDS condition. Note that any minimum CPDS contains at least one of the vertex of the form $(0,\,i)$ or $(j,\,0);$ otherwise, the connected condition fails. Suppose $X$ is a minimum CPDS of $K_{1,\,x}\square K_{1,\,y}$ with size $x-1.$ Since $\vert X\vert =x-1,$ the vertices in at least two $^iK_{1,\,y}$-fiber and two $K_{1,\,x}^j$-fiber do not belong to $X.$ WLOG let $i\in\{x-1,\,x\}$ and $j\in \{y-1,\,y\}.$ Let $Y=\{(a,\,b): a\in\{x-1,\,x\}\,\,\mbox{and}\,\,b\in\{y-1,\,y\} \}.$ It is clear that the vertices in  $Y$ are monitored only by propagation set. But it is not possible as $\vert N((0,\,b))\cap Y\vert > 1$ and $\vert N((a,\,0))\cap Y\vert > 1.$ Which is a contradiction for $\vert X\vert=x-1.$ Hence, $\gamma_{P,\,c}(K_{1,\,x}\square K_{1,\,y})=x.$
\end{proof} 

\begin{theorem}
Let the order of two graphs $G$ and $H$ be at least four and let $\gamma(G)=1.$ $Z_c(H)=2$ if and only if $\gamma_{P,c}(G \square H)=2.$
\end{theorem}
\begin{proof}
By hypothesis and Theorem \ref{upcpdczfs}, $\gamma_{P,c}(G \square H)\leq 2.$ Also, $\gamma_{P,c}(G \square H) > 1$ as $Z_c(H)=2.$ Hence $\gamma_{P,c}(G \square H) = 2.$

Conversely, assume that $\gamma(G)=1$ and $\gamma_{P,c}(G\square H)=2.$ By our assumption, it is clear that $H\not\cong P_m.$ Let $v$ be a universal vertex of $G$ and let $X$ be a CPDS for $G\square H.$ If $(a,\,b)$ and $(c,\,d)$ are the vertices in $X,$ then $a=c=v$ and $b\neq d$ as $\langle X \rangle \cong K_2;$ otherwise $a\neq b$ and $b=d,$ then the vertices in $G \square H$ cannot  be observed by  propagation as $H\not\cong P_m.$ Consequently, propagation occurs from one $G$-fiber to another $G$-fiber only if $Z_c(H)\leq 2.$ Since $H\not\cong P_m,$ $Z_c(H) > 1.$ Thus, $Z_c(H)=2.$
\end{proof}

\begin{theorem}
Let $\gamma(G)=1$ and let $H=G\circ \overline{K_n}.$ For $n,\,m\geq 2,\,\gamma_{P,\,c}(H\square P_m)=2.$
\end{theorem}
\begin{proof}
It is easy to observe that if $\gamma(G)=1,$ then $\gamma(G\circ \overline{K_n})=2$ for all integer $n\geq 2.$ That is, $\gamma_c(H)=2.$ By Theorem \ref{upcpdczfs}, we have $\gamma_{P,\,c}(H\square P_m)\leq 2$ as $Z_c(P_m)=1.$ On the other hand, $\gamma_{P,\,c}(H\square P_m)> 1$ as $\gamma(H)\neq 1.$ Thus, $\gamma_{P,\,c}(H\square P_m)=2.$
\end{proof}

\section{The Tensor Product}

Throughout this section, for a graph $G$ and $H,$ let $V(G)=\{u_1,\,u_2,\,\dots,\,u_a\}$ and $V(H)=\{v_1,\,v_2,\,\dots,\,v_b\}.$ Let $U_i=u_i\times V(H)$ and $V_j=V(G)\times v_j.$ Then $V(G\times H)=\{\bigcup_{i=1}^{a}U_i\}=\{\bigcup_{j=1}^{b}V_j\}.$ The sets $U_i$ and $V_j$ are called the $i^{th}$-row and $j^{th}$-column of the graph $G\times H,$ respectively. 
 
The following theorem is proved for power domination number $G\times H$ but it is true for connected power domination number of $G\times H$ also.   
\begin{theorem}\cite{skp18} \label{cpdntp=1}
If $\gamma_P(G\times H)=\gamma_{P,\,c}(G\times H)=1,$ then $G$ or $H$ is isomorphic to $K_2.$
\end{theorem}

\begin{theorem}
Let $G$ and $H$ be two non-bipartite graphs with at least two universal vertices. Then $\gamma_{P,\,c}(G\times H)= 2.$
\end{theorem}
\begin{proof}
Let $\{u_1,\,u_2\}$ and $\{v_1,\,v_2\}$ be universal vertices of the graphs $G$ and $H,$ respectively. Consider the set $X=\{(u_1,\,v_1),\,(u_2,\,v_2)\} \subset V(G\times H).$ Clearly, $\langle X \rangle \cong K_2.$ Since $u_1$ and $v_1$ are the universal vertices of the graphs $G$ and $H,$ respectively, the vertex $(u_1,\,v_1)$ dominates the vertices in the set $\{\bigcup_{i=2}^a(U_i\setminus(u_i,\,v_1))\}.$ The vertex $(u_2,\,v_2)$ dominates the vertices in the set $(V_1\setminus(u_1,\,v_2))\cup\{\bigcup_{j=3}^b (V_j\setminus (u_2,\,v_j))\}$ as $u_2$ and $v_2$ are the universal vertices of the graphs $G$ and $H,$ respectively. Hence, the only unmonitored vertices of the graph $G\times H$ are $(u_1,\,v_2)$ and $(u_2,\,v_1).$ These vertices are monitored by the propagation step as $\vert N(u_1,\,v_2)\setminus X_1\vert =\vert N(u_2,\,v_1)\setminus X_1\vert = 1.$ Thus, $\gamma_{P,\,c}(G\times H)\leq 2.$ By Theorem \ref{cpdntp=1}, we have  $\gamma_{P,\,c}(G\times H) \neq 1.$ Therefore, $\gamma_{P,\,c}(G\times H)= 2.$
\end{proof}

\begin{corollary}\label{ctp1}
\begin{enumerate}
\item[]
\item For $m,\,n\geq 3,\,\gamma_{P,\,c}(K_m\times K_n)=\gamma_{P}(K_m\times K_n)=2.$
\item For $a\geq 1$ and $b\geq 1,\,\gamma_{P,\,c}(K_{1,\,1,\,m_1,\,m_2,\dots,\,m_a}\times K_{1,\,1,\,n_1,\,n_2,\dots,\,n_b})=$
\item[] $\gamma_{P}(K_{1,\,1,\,m_1,\,m_2,\dots,\,m_a}\times K_{1,\,1,\,n_1,\,n_2,\dots,\,n_b})=2.$
\end{enumerate}
\end{corollary}

\begin{theorem}\label{cpdsgtimeskx,y}
Let $G$ be a non-bipartite graph. For $2\leq x\leq y,\,\gamma_{P,c}(G\times K_{x,\,y})=\gamma_c(G\times K_2).$ 
\end{theorem}

\begin{proof}
Let the bipartition of $K_{x,\,y}$ be $A=\{a_1,\,a_2,\,\dots,\,a_x\}$ and $B=\{b_1,\,b_2,\,\dots,\,b_y\}$ and let $V(G)=\{u_1,\,u_2,\,\dots,\,u_t\}.$ Clearly, $G\times K_{x,\,y}$ is a bipartite graph with bipartition $V_A$ and $V_B,$ where $V_A = V(G) \times A$ and $V_B= V(G) \times B.$ Let $U_i^A=u_i\times A$ and $U_i^B=u_i\times B.$ Then $V(G\times K_{x,\,y}) = V_A \cup V_B= \{\bigcup_{i=1}^t U_i^A\}\cup \{\bigcup_{i=1}^t U_i^B\}.$ Observe that, if $u_iu_j\in E(G),$ then $\langle U_i^A\cup U_j^B\rangle \cong \langle U_j^A\cup U_i^B \rangle\cong K_{x,\,y}.$

 Let $X$ be a minimum connected dominating set of $G\times K_2.$ Now we claim that $X$ is CPDS of $G\times K_{x,\,y}.$ If $(u_i,\,a_i)$ dominates $(u_j,\,b_1),$ then $(u_i,\,a_i)$ dominates all the vertices in $U_j^B$ as $\langle U_i^A\cup U_j^B\rangle \cong K_{x,\,y}.$ Further, each vertex in $G\times K_2$ is adjacent to at least one of the vertices in $X.$ Consequently, $X$ is connected dominating set of $G\times K_{x,\,y}$ and hence $X$ is a CPDS of $G\times K_{x,\,y}.$ From this we have $\gamma_{P,c}(G\times K_{x,\,y})\leq \gamma_c(G\times K_2).$

Assume that $X$ is a minimum CPDS of $G\times K_{x,\,y}$ with $\vert X \vert < \gamma_c(G\times K_2).$ Then we can find $i$ or $j$ such that the vertex $(u_i,\,a_1)$ or $(u_j,\,b_1)$ is not dominated by the vertices in $X.$ This implies that all the vertices in $U_i^A$ or $U_j^B$ are monitored only by  propagation step (not dominating step). But it is not possible as $U_i^A=x\geq 2$ or $U_j^B=y\geq 2.$ Hence, $\gamma_{P,c}(G\times K_{x,\,y})=\gamma_c(G\times K_2).$
\end{proof}

In fact, from the proof of the above theorem, it is easy to observe that $\gamma_{P,c}(G\times K_{x,\,y})= \gamma_{c}(G\times K_{x,\,y})$ for $2\leq x\leq y.$ This observation is used in the proof of the following theorem.

\begin{theorem} \label{gtimeskmn}
Let $G$ be a non-bipartite graph with at least two universal vertices. 
Then $\gamma_{P,c}(G\times K_{x,\,y})= 
\left\{ \begin{array}{rl}
1;& \mbox{if $G \cong C_3$ and $x=y=1,$}\\
2;& \mbox{if $G \not\cong C_3$ and $x=y=1,$}\\
3;& \mbox{if $x=1$ and $y\geq 2,$}\\
4;& \mbox{if $x,\,y\geq 2.$}
\end{array}\right.$
\end{theorem}

\begin{proof}
Consider the vertex set of $G\times K_{x,\,y}$ is as in Theorem \ref{cpdsgtimeskx,y}. Let $u_1$ and $u_2$ be two universal vertices of $G.$

First we complete the proof for $x=y=1.$ If $G\cong C_3,$ then $G\times K_2\cong C_6$ and hence $G\times K_2=1.$ Now we assume that $G\not\cong C_3.$ Let $X=\{(u_1,\,a_1),\,(u_2,\,b_1)\}.$ The vertices $(u_1,\,a_1)$ and $(u_2,\,b_1)$ dominates the vertices in $V_B\setminus (u_1,\,b_1)$ and $V_A\setminus (u_2,\,a_1),$ respectively. The vertices $(u_1,\,b_1)$ and $(u_2,\,a_1)$ are monitored by the propagation step as $\vert N((u_1,\,b_1))\setminus X_1\vert= \vert N((u_2,\,b_1))\setminus X_1\vert=1.$ Hence, $\gamma_{P,\,c}(G\times K_2) \leq 2.$ Since $G$ has two universal vertices, minimum degree of $G$ is at least two and two vertices have degree $t-1.$ As a consequence $\gamma_{P,\,c}(G\times K_2) \neq 1.$ Thus, $\gamma_{P,\,c}(G\times K_2) = 2.$

Now we consider $x=1$ and $y\geq 2.$ For this, let $X=\{(u_1,\,a_1),\,(u_2,\,b_1),\, (u_3,\,a_1)\}.$ The set $X$ dominates all the vertices of $G\times K_{1,\,y}$ except $(u_2,\,a_1).$ This vertex is observed by the propagation step and hence $\gamma_{P,\,c}(G\times K_{1,\,y})\leq 3.$ To prove the equality, assume that $\gamma_{P,\,c}(G\times K_{1,\,y})=2.$ Then the CPDS contains two vertices, namely, $X=\{(u_i,\,a_1),\,(u_j,\,b_m)\},$ where $i\neq j.$ WLOG we assume that $i=1$ and $j=2$ as this choice of $i$ and $j$ dominates maximum number of vertices of $G\times K_{1,\,y}.$ The vertices which are dominated by the vertices in $X$ are the vertices in $U_1^B$ and the vertex $(u_2,\,a_2.)$ Since $\vert U_1^B\vert=y\geq 2,$ propagation step from $(u_i,\,a_1)\in V^A$ to the vertices in $U_1^B$ is not possible. This implies that $\gamma_{P,\,c}(G\times K_{1,\,y})\neq 2.$ Thus, $\gamma_{P,\,c}(G\times K_{1,\,y})=3.$

Let $2\leq x\leq y.$ Recall that $\gamma_{P,c}(G\times K_{x,\,y})= \gamma_{c}(G\times K_{x,\,y})$ for $2\leq x\leq y.$ Form this, it is enough to find $\gamma_{c}(G\times K_{x,\,y}).$ Let $X=\{(u_1,\,a_1),\,(u_2,\,b_1),\,(u_3,\,a_1),\,(u_1,\,b_1)\}.$ Clearly, the vertices in the set $X$ dominate all the vertices $G\times K_{x,\,y}$ and $\langle X\rangle \cong P_4$ and hence $\gamma_{c}(G\times K_{x,\,y})\leq 4.$ Since $G\times K_{x,\,y}$ is bipartite, connected subgraph induced by any three vertices of $G\times K_{x,\,y}$ is isomorphic to $P_3.$ Clearly, the end vertices of $P_3$ belong to either $V^A$ or $V^B.$ We assume that the end vertices of $P_3$ belong to $V^A.$ Then the two degree vertex belongs to $V^B.$ Let the two degree vertex be $(u_i,\,b_j).$ Clearly, this vertex does not dominates the vertices in the set $U_i^A.$ Consequently, three vertices do not form the connected dominating set. Therefore, $\gamma_{c}(G\times K_{x,\,y})\geq 4.$ 
\end{proof}

\begin{theorem} \label{gtimesmul}
Let $G$ be a graph with at least two universal vertices. For $k\geq 3$ and $1\leq m_1 \leq m_2 \leq \dots \leq m_k,$ 

$\gamma_{P,\,c}(G\times K_{m_1,\,m_2,\,\dots,\,m_k})= 
\left\{ \begin{array}{rl}
2;& \mbox{if $m_1=m_2=1,$}\\
3;& \mbox{otherwise}
\end{array}\right.$
\end{theorem}
\begin{proof}
Let $V(G)=\{u_1,\,u_2,\,\dots,\,u_t\}.$ For $1\leq i \leq k,$ let $V_i=\{a_1^i,\, a_2^i,\,\dots,\,a_{m_i}^i\}$ denote the $i^{th}$ partite set of the graph $K_{m_1,\,m_2,\,\dots,\,m_k}$ with size $m_i.$ Let $U_i=\{\bigcup_{j=1}^k U_i^{V_j}\},$ where $U_i^{V_j}=u_i\times V_j.$ Then $V(G\times K_{m_1,\,m_2,\,\dots,\,m_k}) = \bigcup_{i=1}^i U_i= \bigcup_{i=1}^i\{\bigcup_{j=1}^k U_i^{V_j}\}.$ Let the universal vertices of $G$ be $u_1$ and $u_2.$ Let $H=G\times K_{m_1,\,m_2,\,\dots,\,m_k}.$

If $m_1=m_2=1,$ then the result follows by Corollary \ref{ctp1}. Now we assume that $m_2\geq 2.$ Consider the set $X=\{(u_1,\,a_1^1),\,(u_2,\,a_1^2),\,(u_3,\,a_1^3)\}.$ The vertices in $V(H)\setminus (U_2^{V_1}\cup U_1^{V_2})$ are dominated by the vertices in $\{(u_1,\,a_1^1),\,(u_2,\,a_1^2)\}$ and the vertices in $U_2^{V_1}\cup U_1^{V_2}$ are dominated by the vertex $(u_3,\,a_1^3).$ Hence $X$ is CPDS of  $H.$ This gives $\gamma_{P,\,c}(H)\leq 3.$ To obtain the reverse inequality, we claim that any set $X$ contains two vertices of $H$ is not a CPDS of $H.$ Let $X=\{(u_i,\,a_1^x),\,(u_j,\,a_1^y)\}.$ Then $X_1=N[X].$ Clearly, the set $X_1$ does not contain the vertices in the set $U_i^{V_x}\cup U_j^{V_y}.$ The propagation step from any vertex in $X_1$ to any vertex in  $U_i^{V_x}\cup U_j^{V_y}$ is not possible as $\vert U_i^{V_x}\vert$ and $\vert U_j^{V_y}\vert$ are at least two. Consequently, $\gamma_{P,\,c}(H) >2.$ Hence, $\gamma_{P,\,c}(G\times K_{m_1,\,m_2,\,\dots,\,m_k})= 3.$
\end{proof}

\begin{theorem}
For $t\geq 3,\,k\geq 3,\,1\leq n_1\leq n_2\leq \dots\leq n_t$ and $1\leq m_1 \leq m_2 \leq \dots \leq m_k,$ we have 

$\gamma_{P,\,c}(K_{n_1,\,n_2,\,\dots,\,n_t}\times K_{m_1,\,m_2,\,\dots,\,m_k})= 
\left\{ \begin{array}{rl}
2;& \mbox{if $n_1=n_2=1$ and $m_1=m_2=1,$}\\
3;& \mbox{otherwise}
\end{array}\right.$
\end{theorem}
\begin{proof}
Let $U_i=\{u_1^i,\,u_2^i,\,\dots,\,u_{n_i}^i\}$ and $V_j=\{v_1^j,\,v_2^j,\,\dots,\,v_{m_j}^j\},$ respectively, denote the partite sets of the graphs $K_{n_1,\,n_2,\,\dots,\,n_t}$ and $K_{m_1,\,m_2,\,\dots,\,m_k}$ of sizes $n_i$ and $m_j.$ If $n_1=n_2=m_1=m_2=1,$ then the result is immediate by Corollary \ref{ctp1}. Now we assume that $n_2\geq 2.$ Consider the set $X=\{(u_1^1,\,v_1^1),\,(u_1^2,\,v_1^2),\,(u_1^3,\,v_1^3)\}.$ The vertices in $X$ dominates all the vertices of $K_{n_1,\,n_2,\,\dots,\,n_t}\times K_{m_1,\,m_2,\,\dots,\,m_k}$ and hence $\gamma_{P,\,c}(K_{n_1,\,n_2,\,\dots,\,n_t}\times K_{m_1,\,m_2,\,\dots,\,m_k})\leq 3.$ By employing similar argument as in Theorem \ref{gtimesmul}, we can conclude that $\gamma_{P,\,c}(K_{n_1,\,n_2,\,\dots,\,n_t}\times K_{m_1,\,m_2,\,\dots,\,m_k})\geq 3.$
\end{proof}

\bmhead{Acknowledgment}
The first author is supported by Dr. D. S. Kothari Postdoctoral Fellowship, University Grand Commission, Government of India, New Delhi through the Grand No.F.4-2/2006(BSR)/MA/20-21/0067.

\end{document}